\documentclass[11pt]{article}

\usepackage{amsmath, amsthm, amssymb}
\usepackage{graphicx} 
\usepackage{comment}
\newtheorem{theorem}{Theorem}
\newtheorem{lemma}{Lemma}

\newtheorem{claim}{Claim}
\newtheorem{corollary}{Corollary}

\newtheorem{definition}{Definition}

\author{Aistis Atminas\thanks{DIMAP and Mathematics Institute, University of Warwick, Coventry CV4 7AL, UK. E-mail: A.Atminas@warwick.ac.uk.}
\and Vadim V.~Lozin\thanks{DIMAP and Mathematics Institute, University of Warwick, Coventry CV4 7AL, UK. E-mail: V.Lozin@warwick.ac.uk.
Research of this author was supported by the Centre for Discrete Mathematics 
and Its Applications (DIMAP), University of Warwick, and by EPSRC grant EP/I01795X/1.} 
\and Igor Razgon\thanks{Department of Computer Science and Information Systems,
Birkbeck, University of London, E-mail: igor@dc.bbk.ac.uk}}

\date{}
\title{Graphs without large bicliques and \\ well-quasi-orderability  by the induced
       subgraph relation} 

\begin{document}

\maketitle

\begin{abstract}
Recently,  Daligault, Rao and Thomass\'e asked in \cite{DRT10} if every hereditary 
class which is well-quasi-ordered by the induced subgraph relation is of bounded clique-width.
There are two reasons why this questions is interesting. First, it connects two seemingly unrelated
notions. Second, if the question is answered affirmatively, this will have a strong algorithmic 
consequence. In particular, this will mean (through the use of
Courcelle theorem \cite{CorMakRotics}), that any problem definable in Monadic Second Order Logic can be solved in
a polynomial time on any class well-quasi-ordered by the induced subgraph relation.

In the present paper, we answer this question affirmatively for graphs without large bicliques.
Thus the above algorithmic consequence is true, for example, for classes of graphs of bounded degree. 
\end{abstract}

{\it MSC codes:} 05C75 Structural characterization of families of graphs; 
                 05C85 Graph algorithms.

\section{Introduction}
Well-quasi-ordering is a highly desirable property and a frequently discovered
concept in mathematics and theoretical computer science \cite{Finkel,Kruskal}. 
One of the most remarkable recent results in this area is the proof of Wagner's conjecture
stating that the set of all finite graphs is well-quasi-ordered by the minor relation
\cite{minor-wqo}. However, the subgraph or induced subgraph relation is not a well-quasi-order.
On the other hand, each of these relations may become a well-quasi-order when restricted to graphs 
with some special properties. In this paper, we study well-quasi-orderability of graphs with 
hereditary properties.

A {\it  graph property} (or a {\it class of  graphs}) is a set of graphs closed under isomorphism.
A property is {\it hereditary} if it is closed under taking induced subgraphs.
It is well-known (and not difficult to see) that a graph property $X$ is hereditary if and only if 
$X$ can be described in terms of forbidden induced subgraphs. More formally, $X$ is hereditary 
if and only if there is a set $M$ of graphs such that no graph in $X$ contains any graph from $M$
as an induced subgraph. We call $M$ the set of {\it forbidden induced subgraphs} for  $X$ and 
say that the graphs in $X$ are $M$-free.  

Of our particular interest in this paper are graphs \emph{without large bicliques}.
We say that the graphs in a hereditary class $X$ are \emph{without large bicliques} if 
there is a natural number $t$ such that no graph in $X$ contains $K_{t,t}$ as a (not necessarily induced) subgraph.
Equivalently, there are $q$ and $r$ such $K_{q,q}$ and $K_r$ appear in the set of forbidden induced subgraphs of $X$.
According to \cite{Zaran}, these are precisely graphs with a subquadratic number of edges. This family of properties 
includes many important classes, such as graphs of bounded vertex degree, of bounded tree-width, all proper minor closed graph classes.
In all these examples, the number of edges is bounded by a linear function in the number of vertices
and all of the listed properties are rather small (see e.g. \cite{minor-closed-small} for the number of graphs in 
proper minor closed graph classes). In the terminology of \cite{speed}, they all are at most factorial.
In fact the family of classes without large bicliques is much richer and contains classes with a superfactorial speed of growth, 
such as projective plane graphs (or more generally $C_4$-free bipartite graphs), in which case the number of edges is $\Theta(n^{\frac{3}{2}})$.

Recently,  Daligault, Rao and Thomass\'e asked in \cite{DRT10} if every hereditary 
class which is well-quasi-ordered by the induced subgraph relation is of bounded clique-width.
There are two reasons why this questions is interesting. First, it connects two seemingly unrelated
notions. Second, if the question is answered affirmatively, this will have a strong algorithmic consequence.
In particular, this will mean (through the use of
Courcelle theorem \cite{CorMakRotics}), that any problem definable in Monadic Second Order Logic can be solved in
a polynomial time on any class well-quasi-ordered by the induced subgraph relation.

In the present paper, we answer this question affirmatively for graphs without large bicliques. 
More precisely, we prove that if a class $X$ without large bicliques is well-quasi-ordered by the induced subgraph
relation, then the graphs in $X$ have bounded treewidth, i.e. there is a constant $c$ such that the treewidth of any
graph in $X$ is at most $c$. Since treewidth and cliquewidth of graphs without large bicliques are known
to be equivalent in the sense that one is bounded if and only if the other is \cite{GurWan}, the result affirmatively 
answers the question in \cite{DRT10} for graphs without large bicliques.  Thus the above algorithmic consequence
is confirmed e.g. for classes of graphs of bounded degree.

In order to establish the main result (Theorem~\ref{thm:main-wqo}), we define in Section~\ref{sec:basic} an infinite family 
of graphs pairwise incomparable by the induced subgraph relation, which we call \emph{canonical graphs}.
The main part of the proof of Theorem \ref{thm:main-wqo} is a combinatorial result stating that a graph without large bicliques
and having a large treewidth has a large induced canonical graph. A consequence of this result is that
if a class $X$ without large bicliques has unbounded treewidth, then $X$ contains an infinite subset of
canonical graphs, i.e. an infinite antichain. This implies that classes of graphs without large bicliques
that are well quasi-ordered by the induced subgraph relation must have bounded treewidth.

To prove the main theorem, we first prove an auxiliary result (Theorem~\ref{thm:main})
stating that if a graph without large bicliques has a long path, it also has a long induced path. 
We note that this auxiliary theorem is sufficient to establish the main result of the paper if we confine
ourselves to hereditary classes with a \emph{finite} number of forbidden induced subgraphs. 

All preliminary information related to the topic of the paper can be found in Section~\ref{sec:basic}.
In Sections~\ref{sec:paths} and~\ref{sec:wqo} we prove  Theorem~\ref{thm:main}  and Theorem~\ref{thm:main-wqo},
respectively.

\section{Notations and definitions}
\label{sec:basic}

We consider only simple undirected graphs without loops and multiple edges.
An {\it independent set} in a graph is a set of vertices no two of which are adjacent, 
and a {\it clique} is a set of vertices every two of which are adjacent. 
As usual, by $K_n$, $P_n$ and $C_n$ we denote the complete graph, the chordless path and the chordless cycle
on $n$ vertices, respectively, and $K_{n,m}$ is a complete bipartite graph with parts of size $n$ and $m$. 
Sometimes we also refer to $K_n$ as a clique and to $K_{n,m}$ as a biclique.  
If $n=m$ we say that $K_{n,m}$ is a biclique of {\it order} $n$.  

Given a graph $G$ and a subset $U$ of its vertices, the operation of contraction of $U$
into a single vertex $u$ consists in deleting $U$, introducing $u$ and 
connecting $u$ to every vertex of $G$ outside $U$ that has a neighbour in $U$.
If $U$ consists of two adjacent vertices, this operation is called edge contraction.   
 
Let $H$ and $G$ be two graphs. We say that  
\begin{itemize}
\item $H$ is an {\it induced subgraph} of $G$ if $H$ can be obtained from $G$ by vertex deletions,
\item $H$ is a {\it subgraph} of $G$ if $H$ can be obtained from $G$ by vertex deletions and edge deletions,
\item $H$ is a {\it minor} of $G$ if $H$ can be obtained from $G$ by vertex deletions, edge deletions and edge contractions.
\end{itemize}

Throughout the text, whenever we say that $G$ contains $H$, we mean that $H$ is a subgraph of $G$, 
unless we explicitly say that $H$ is an {\it induced} subgraph of $G$ (or $G$ contains $H$ as an {\it induced} subgraph).
If $H$ is not an induced subgraph of $G$, we say that $G$ is $H$-free. 

By $R=R(k,r,m)$, we denote the Ramsey number, i.e. the minimum $R$ such that in every 
colouring of $k$-subsets of an $R$-set with $r$ colours there is a monochromatic
$m$-set, i.e. a set of $m$ elements all of whose $k$-subsets have the same colour.


A binary relation $\le$ on a set $X$ is a {\it quasi-order} if it is reflexive and transitive.
If additionally $\le$ is antisymmetric, then it is a partial order. 
Two elements $x,y\in X$ are said to be incomparable if neither $x\le y$ nor $y\le x$. 
An {\it antichain} in a quasi-order is a set of pairwise incomparable elements.   
A quasi-order $(X,\le)$ is a {\it well-quasi-order} if $X$
contains no infinite strictly decreasing sequences and no infinite antichains. 

According to the celebrated Graph Minor Theorem of Robertson and Seymour, 
the set of all graphs is well-quasi-ordered by the graph minor relation \cite{minor-wqo}. 
This, however, is not the case for the more restrictive relations such as subgraph or induced subgraph.
Consider for instance the graphs $H_1,H_2,\ldots$, where $H_i$ is the graph represented in Figure~\ref{fig:H}.
It is not difficult to see that this sequence creates an infinite antichain with respect to both 
subgraph and induced subgraph relations.

By connecting two vertices of degree one having a common neighbour in $H_i$,
we obtain a graph represented on the left of Figure~\ref{fig:H'H''}. 
Let us denote this graph by $H'_i$. By further connecting the other 
pair of vertices of degree one we obtain the graph $H''_i$  
represented on the right of Figure~\ref{fig:H'H''}.
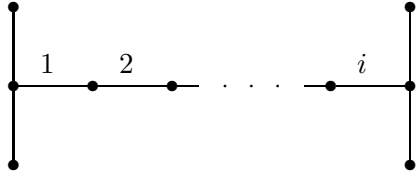
\begin{figure}[ht]
\begin{center}
\begin{picture}(180,70)
\put(0,35){\circle*{4}} \put(30,35){\circle*{4}}
\put(60,35){\circle*{4}} \put(80,35){\circle*{1}}
\put(90,35){\circle*{1}} \put(100,35){\circle*{1}}
\put(120,35){\circle*{4}} \put(150,35){\circle*{4}}
\put(0,65){\circle*{4}} \put(0,5){\circle*{4}}
\put(150,65){\circle*{4}} \put(150,5){\circle*{4}}
\put(2,35){\line(1,0){26}} \put(32,35){\line(1,0){26}}
\put(62,35){\line(1,0){8}} \put(110,35){\line(1,0){8}}
\put(122,35){\line(1,0){26}} \put(0,37){\line(0,1){26}}
\put(0,33){\line(0,-1){26}} \put(150,37){\line(0,1){26}}
\put(150,33){\line(0,-1){26}} \put(10,40){1} \put(40,40){2}
\put(130,40){$i$}
\end{picture}
\caption{The graph $H_{i}$}
\label{fig:H}
\end{center}
\end{figure}
\begin{figure}[ht]
\begin{center}
\begin{picture}(170,65)
\put(0,5){\circle*{4}}
\put(0,45){\circle*{4}} 
\put(20,25){\circle*{4}} 
\put(50,25){\circle*{4}}
\put(70,25){\circle*{1}} 
\put(80,25){\circle*{1}}
\put(90,25){\circle*{1}}
\put(110,25){\circle*{4}} 
\put(140,25){\circle*{4}}
\put(22,25){\line(1,0){26}} 
\put(18,23){\line(-1,-1){16}}
\put(18,27){\line(-1,1){16}}
\put(142,23){\line(1,-1){16}}
\put(142,27){\line(1,1){16}}
\put(52,25){\line(1,0){8}} 
\put(100,25){\line(1,0){8}}
\put(112,25){\line(1,0){26}}
\put(160,5){\circle*{4}}
\put(160,45){\circle*{4}}
\put(160,7){\line(0,1){36}}
\put(33,30){1}
\put(123,30){$i$}
\end{picture}
\begin{picture}(170,65)
\put(0,5){\circle*{4}}
\put(0,45){\circle*{4}} 
\put(20,25){\circle*{4}} 
\put(50,25){\circle*{4}}
\put(70,25){\circle*{1}} 
\put(80,25){\circle*{1}}
\put(90,25){\circle*{1}}
\put(110,25){\circle*{4}} 
\put(140,25){\circle*{4}}
\put(0,7){\line(0,1){36}}
\put(22,25){\line(1,0){26}} 
\put(18,23){\line(-1,-1){16}}
\put(18,27){\line(-1,1){16}}
\put(142,23){\line(1,-1){16}}
\put(142,27){\line(1,1){16}}
\put(52,25){\line(1,0){8}} 
\put(100,25){\line(1,0){8}}
\put(112,25){\line(1,0){26}}
\put(160,5){\circle*{4}}
\put(160,45){\circle*{4}}
\put(160,7){\line(0,1){36}}
\put(33,30){1}
\put(123,30){$i$}
\end{picture}
\caption{Graphs $H'_{i}$ and $H''_{i}$}
\label{fig:H'H''}
\end{center}
\end{figure}
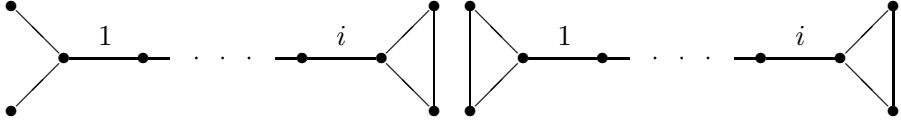

We call any graph of the form $H_i$, $H'_i$ or $H''_i$ an $H$-graph. Furthermore, 
we will refer to $H''_i$ a {\it tight} $H$-graph and to $H'_i$ a {\it semi-tight} $H$-graph.
In an $H$-graph, the path connecting two vertices of degree 3  will be called the {\it body} of the graph,
and the vertices which are not in the body the {\it wings}. 

Following standard graph theory terminology, we call a chordless cycle of length at least four a {\it hole}.
Let us denote by 
\begin{itemize}
\item[$\cal C$] the set of all holes and all $H$-graphs. 
\end{itemize}
It is not difficult to see that any two distinct (i.e. non-isomorphic) graphs in $\cal C$ are incomparable with 
respect to the induced subgraph relation. In other words,
\begin{claim}\label{claim:antichain}
$\cal C$ is an antichain with respect to the induced subgraph relation.
\end{claim}
Moreover, from the poof of Theorem~\ref{thm:main-wqo} we will see that 
for classes of graphs without large bicliques which are of unbounded tree-width this antichain is unavoidable,
or {\it canonical}, in the terminology of \cite{Ding09}. Suggested by this observation, 
we introduce the following definition.

\begin{definition}
The graphs in the set $\cal C$ will be called {\sc canonical}.
\end{definition}

The {\it order} of a canonical graph $G$ is either the number of its vertices, if $G$ is a hole,
or the the number of vertices in its body, if $G$ is an $H$-graph.


\section{Long paths in graphs without large bicliques}
\label{sec:paths}


In this section, we prove that graphs without large bicliques containing 
a large path also contain a large induced (i.e. chordless) path.
We start with the following auxiliary result, where by 
$P(r,m)$, we denote the minimum $n$ such that in every 
colouring of the elements of an $n$-set with $r$ colours there exists 
a subset of $m$ elements of the same colour (the pigeonhole principle).

\begin{lemma}\label{lem:grid}
For each $p$ and $q$ there is a number $C=C(p,q)$ such that 
whenever a graph $G$ contains two families of sets $\mathcal{A}=\{V_1, V_2, \ldots, V_{C}\}$ 
and $\mathcal{B}=\{W_1, W_2, \ldots, W_{C}\}$ with all sets being disjoint of size $p$ 
and with at least one edge between every two sets $V_i\in \mathcal{A}$ and $W_j\in\mathcal{B}$,
then $G$ contains a biclique $K_{q,q}$.
\end{lemma}

\begin{proof} 
We define $r:=P(p^q,q)$ and $C(p,q):=P(p^{r},q)$ and consider an arbitrary collection $A$ 
of $r$ sets from $\cal A$. Since each set in $\cal B$ has a neighbour in each set in $\cal A$,
the family of the sets in $\cal B$ can be coloured with at most $p^r$ colours so that all sets 
of the same colour have a common neighbour in each of the $r$ chosen sets of collection $A$. 
By the choice of $C(p,q)$, one of the colour classes contains a collection $B$ of at least $q$ sets. 
For each set in $A$, we choose a vertex which is a common neighbour for all sets in $B$ and denote 
the set of $r$ chosen vertices by $U$. The vertices of $U$ can be coloured with at most $p^q$ colours 
so that all vertices of the same colour have a common neighbour in each of the $q$ sets of collection $B$. 
By the choice of $r$, $U$ contains a colour class $U_1$ of least $q$ vertices. For each set in $B$, 
we choose a vertex which is a common neighbour for all vertices of $U_1$ and denote the set of $q$ 
chosen vertices by $U_2$. Then $U_1$ and $U_2$ form a biclique $K_{q,q}$. 
\end{proof}

\begin{theorem}\label{thm:main2}
For every $s$ and $q$ there is a number $Y=Y(s,q)$ such that every graph with a path of length at least $Y$ 
contains either a path $P_s$ as an induced subgraph or a biclique $K_{\lfloor q/2 \rfloor, \lceil q/2 \rceil}$ 
as a (not necessarily induced) subgraph. 
\end{theorem}

\begin{proof}
We use induction on $s$ and $q$. For $s=1$ and  arbitrary $q$ or for $q=1$ and arbitrary  $s$, 
we can take $Y(s,q)=1$. So assume $s>1$ and $q>1$. Let $t=Y(s, q - 1)$ and $k=Y(s - 1, C(t, q))$.
Both numbers must exist by the induction hypothesis. 

Consider a graph $G$ with a path $P=v_1v_2 \ldots v_{kt}$ on $kt$ vertices
and  split $P$ into $k$ subpaths of $t$ vertices each. We denote the vertices of the $i$-th subpath by 
$V_i$
and form a graph $H$ on $k$ vertices $\{h_1, h_2,  \ldots , h_k\}$ in which $h_ih_j$ is an edge if and only if 
there is an edge in $G$ joining a vertex of $V_i$ to a vertex of $V_j$. 
Since $h_i$ is joined to $h_{i+1}$ for each $i=1,\ldots,k-1$, the graph $H$ has a path 
on $k$ vertices, and since $k=Y(s - 1, C(t, q))$, it has either an induced path on $s-1$ vertices or a biclique of order $C(t,q)$. 
In the graph $G$, the latter case corresponds to two families of $C(t,q)$ pairwise disjoint subsets with $t$ vertices in each subset 
and with an edge between any two subsets from different families. 
Therefore, Lemma~\ref{lem:grid} applies proving that $G$ contains a biclique $K_{q,q}$. 

Now assume  $H$ contains an induced path $P_{s-1}$. In the graph $G$, this path corresponds to an ordered sequence of subsets 
$V_{i_1}, V_{i_2}, \ldots, V_{i_{s-1}}$ with edges appearing only between consecutive subsets of the sequence. 
Therefore, in the subgraph of $G$ induced by these subsets, 
any vertex $v$ in $V_{i_1}$ is of distance at least $s - 2$ from any vertex $u$ in $V_{s-1}$.
If the distance between $v$ and $u$ is $s-1$, the graph $G$ has an induced path $P_s$ and we are done.
So, assume the distance between any two vertices of $V_{i_1}$ and $V_{i_{s-1}}$
is exactly $s - 2$, and consider a path with exactly one vertex $w_p$ in each $V_{i_p}$. 

If vertex $w_1$ has a neighbour $w\in V_{i_1}$ which is not adjacent to $w_2$, then
$ww_1w_2 \ldots w_{s-1}$ is an induced path $P_s$ and we are done. Therefore, we must assume that 
$w_2$ is adjacent to every vertex of $V_{i_1}$, since this set induces a connected subgraph.    
As the size of $V_{i_1}$ is $t=Y(s, q - 1)$, it contains either an induced path $P_s$, in which case we are done, 
or a biclique $K_{\lfloor (q-1)/2 \rfloor, \lceil (q-1)/2 \rceil}$. In the latter case, the biclique together with 
$w_2$ form a biclique of the desired size $K_{\lfloor q/2 \rfloor, \lceil q/2 \rceil}$, 
so we are done as well. This completes the proof.
\end{proof}

Taking into account that a large biclique gives rise either to a large induced biclique or a large clique,
Theorem~\ref{thm:main2} can also be restated as follows. 

\begin{theorem}\label{thm:main}
For every $s$, $t$, and $q$, there is a number $Z=Z(s,t,q)$ such that every 
graph with a path of length at least $Z$ contains either $P_s$ or $K_t$ or $K_{q,q}$
as an induced subgraph. 
\end{theorem}

It turns out that Theorem \ref{thm:main} is sufficient to establish the main claim
of the paper if we confine ourselves to \emph{finitely defined} classes of graphs,
i.e. those defined by forbidding finitely many induced subgraphs. Indeed,
a finitely defined class $X$ is well-quasi-ordered by the induced subgraph relation only if a path $P_s$ for some $s$ is forbidden for $X$, 
since otherwise the class contains infinitely many cycles, i.e. an infinite antichain. 
Therefore, by Theorem \ref{thm:main}, if graphs in $X$ are $(K_t,K_{q,q})$-free,
then they do not contain $P_{Z}$ as a (not necessarily induced) subgraph with $Z=Z(s,t,q)$.

On the other hand, it is well-known \cite{Fellows89} that large treewidth of a graph implies the existence of a large path. 
Put it differently, a bound on the length of a path implies a bound on treewidth. Since we know that 
in a finitely defined class, well-quasi-ordered by the induced subgraph relation, the path length is bounded, we conclude that the treewidth is
bounded as well. This gives us the following corollary.

\begin{corollary} \label{col:aux}
Let $X$ be a hereditary subclass of $(K_t,K_{q,q})$-free graphs defined by a finite collection 
of forbidden induced subgraphs. If $X$ is well-quasi-ordered by the induced subgraph relation, 
then $X$ is of bounded treewidth.
\end{corollary}


\section{Main result}
\label{sec:wqo}


The arguments given to justify Corollary~\ref{col:aux} are not applicable to 
hereditary classes defined by infinitely many forbidden induced subgraphs,
because in this case well-quasi-orderability does not necessarily imply 
a bound on the length of a path. Indeed, consider for instance the class
of $(K_{1,3},C_3,C_4,C_5,\ldots)$-free graphs. It consists of linear forests, 
i.e. graphs every connected component of which is a path. This class is well-quasi-ordered
by the induced subgraph relation, but the path length is not bounded in this class.
In order to address this more general situation, in this section we prove the following 
theorem which is the main result of the paper. 

\begin{theorem}\label{thm:main-wqo}
If $X$ is a hereditary subclass of $(K_t,K_{q,q})$-free graphs which is well-quasi-ordered by the induced subgraph relation,
then $X$ has a bounded treewidth. 
\end{theorem}

To prove the theorem, we will show that a large treewidth combined with the absence of large bicliques 
implies the existence of a large induced canonical graph, which is a much richer structural consequence 
than just the existence of a long induced path. An important part of showing the existence of a large
canonical graph is verifying that its body (see Section~\ref{sec:basic} for the terminology) is induced.
This will be done by application of Theorem~\ref{thm:main}. 

A plan of the proof of Theorem~\ref{thm:main-wqo} is outlined in Section~\ref{sec:plan}.
Sections~\ref{sec:tw-rake},~\ref{sec:from},~\ref{sec:dense},~\ref{sec:ltw},~\ref{sec:summary}
contain various parts of the proof.

\subsection{Plan of the proof}
\label{sec:plan}
To prove Theorem~\ref{thm:main-wqo} we will show that graphs of arbitrarily large tree-width 
contain either arbitrarily large bicliques as subgraphs or arbitrarily large canonical graphs
as induced subgraphs. The main notion in our proof is that of a {\it rake-graph}.

A {\it rake-graph} (or simply a {\it rake}) consists of a chordless path,
the {\it base} of the rake, and a number of pendant vertices, called {\it teeth},
each having a private neighbour on the base. The only neighbour of a tooth on
the base will be called the {\it root} of the tooth, and a rake with $k$ teeth will be 
called a $k$-rake. We will say that a rake is $\ell$-{\it dense} if any $\ell$ consecutive 
vertices of the base contain at least one root vertex. An example of a 1-dense 9-rake is 
given in Figure~\ref{fig:rake}.
     
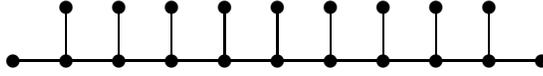
\begin{figure}[ht]
\begin{center}
\begin{picture}(200,50)
\put(0,0){\circle*{5}}
\put(20,0){\circle*{5}}
\put(40,0){\circle*{5}}
\put(60,0){\circle*{5}}
\put(80,0){\circle*{5}}
\put(100,0){\circle*{5}}
\put(120,0){\circle*{5}}
\put(140,0){\circle*{5}}
\put(160,0){\circle*{5}}
\put(180,0){\circle*{5}}
\put(200,0){\circle*{5}}

\put(20,20){\circle*{5}}
\put(40,20){\circle*{5}}
\put(60,20){\circle*{5}}
\put(80,20){\circle*{5}}
\put(100,20){\circle*{5}}
\put(120,20){\circle*{5}}
\put(140,20){\circle*{5}}
\put(160,20){\circle*{5}}
\put(180,20){\circle*{5}}

\put(0,0){\line(1,0){20}}
\put(20,0){\line(1,0){20}}
\put(40,0){\line(1,0){20}}
\put(60,0){\line(1,0){20}}
\put(80,0){\line(1,0){20}}
\put(100,0){\line(1,0){20}}
\put(120,0){\line(1,0){20}}
\put(140,0){\line(1,0){20}}
\put(160,0){\line(1,0){20}}
\put(180,0){\line(1,0){20}}

\put(20,0){\line(0,1){20}}
\put(40,0){\line(0,1){20}}
\put(60,0){\line(0,1){20}}

\put(80,0){\line(0,1){20}}
\put(100,0){\line(0,1){20}}
\put(120,00){\line(0,1){20}}

\put(140,0){\line(0,1){20}}
\put(160,0){\line(0,1){20}}
\put(180,0){\line(0,1){20}}

\end{picture}
\end{center}
\caption{1-dense 9-rake}
\label{fig:rake}
\end{figure}

\medskip  
We will prove Theorem~\ref{thm:main-wqo} through a number of intermediate steps as follows.
\begin{itemize}
\item[1.] In Section~\ref{sec:tw-rake}, we observe that any graph of large tree-width 
contains a rake with many teeth as a subgraph.
\item[2.] In Section~\ref{sec:from} we show that any graph containing 
a rake with many teeth as a subgraph contains either 
\begin{itemize}
\item a {\it dense} rake with many teeth as a subgraph or 
\item a large canonical graph as an {\it induced} subgraph.
\end{itemize}
\item[3.] In Section~\ref{sec:dense} we prove that dense rake
subgraphs necessarily imply either   
\begin{itemize}
\item a large canonical graph as an {\it induced} subgraph or
\item a large biclique as a subgraph.
\end{itemize}
\item[4.] In Section~\ref{sec:ltw}, we summarize the results of the previous sections to show that 
any graph of large tree-width contains 
either   
\begin{itemize}
\item a large canonical graph as an {\it induced} subgraph or
\item a large biclique as a subgraph.
\end{itemize}
\item[5.] In Section~\ref{sec:summary}, we use the result of Step 4 to prove Theorem~\ref{thm:main-wqo}. 
\end{itemize}

\subsection{Rake subgraphs in graphs of large tree-width}
\label{sec:tw-rake}

\begin{lemma}\label{lem:tw-rake}
For any natural $k$, there is a number $f(k)$ such that every graph 
of tree-width at least $f(k)$ contains a $k$-rake as a subgraph.
\end{lemma}

\begin{proof}  
A $k\times k$-grid is a graph with vertices $v_{i,j}$ $1\le i,j,\le k$ 
and edges between $v_{i,j}$ and $v_{i',j'}$ if and only if $|i-i'|+|j-j'|=1$. 
In \cite{RST94}, the authors proved that for each $k$ there 
is a function $f(k)$ such that every graph $G$ of tree-width at least $f(k)$ has 
a $k\times k$-grid as a minor.

Consequently, any graph $G$ of tree-width at least $f(k)$ contains a $k$-rake as a minor.
It follows that the graph $G$ contains a subgraph $H$ from which a $k$-rake can be obtained by contraction operations only.
We deduce that $G$ contains a subgraph $H$, whose vertices admit a partition $V(H)=\cup_{i=1}^k V_i \cup_{i=1}^k V'_i$ into disjoint subsets $V_i$ and $V'_i$ such that $G[V_i]$ and $G[V_i']$ are connected for each $i \in \{1,2, \ldots, k\}$, there is at least one edge with endpoints in both $V_i$ and $V_{i+1}$ for each $i=\{1,2, \ldots, k-1\}$ and there is at least one edge with endpoints in both $V_i$ and $V'_i$ for each $i=\{1,2, \ldots, k\}$.



To finish the proof we show that the graph $H$ contains a $k$-rake as a subgraph. First, for each $i=\{1, 2, \ldots, k-1\}$, let $x_iy_{i+1}$ be an edge with $x_i \in V_i$ and $y_{i+1} \in V_{i+1}$. Then, for each $i=\{2, 3, \ldots, k-1\}$, as $G[V_i]$ is connected,  we can find a path $P_i$ in $G[V_i]$ connecting $y_i$ and $x_i$. We also define $P_1=\{x_1\}$ and $P_k=\{y_k\}$. These paths will constitute the base of the rake and one can attach tooth $t_i$ with root in $P_i$ as follows. If $V(P_i)=V_i$, let $t_i$ be a point in $V'_i$ which is adjacent
to some point in $V_i$. Otherwise, if $V(P_i) \neq V_i$, let $t_i$ be a point in $V_i \backslash V(P_i)$ which has a neighbour in $V(P_i)$ (possible as $G[V_i]$ is connected). Thus $H$, and hence $G$, contains as a subgraph a $k$-rake with base $P_1 \cup P_2 \cup \ldots \cup P_k$ and teeth $\{t_1, t_2, \ldots, t_k\}$.

\end{proof}

\subsection{From rake subgraphs to dense rake subgraphs}
\label{sec:from}

The main result of this section is Lemma~\ref{lem:to-dense-rake} below. Its proof is based 
on the following auxiliary result.

\begin{lemma}\label{lem:shorten}
Let $G$ be a graph containing an $H$-graph $H^*$ (possibly tight or semi-tight) as a subgraph
with the body being induced (i.e. chordless), and let $s\ge 2$ an integer. Then 
\begin{itemize}
\item[(1)] either $G$ contains a path of length $t\in \{2,\ldots,s+1\}$ connecting a left wing of $H^*$ to 
its right wing with all intermediate vertices lying in the body,
\item[(2)]  or $G$ contains an induced canonical subgraph of order at least $s$.  
\end{itemize}
\end{lemma}

\begin{proof}
Let $w'$ be a left wing and $w''$ be a right wing of $H^*$ and $U=\{u_1,\ldots,u_q\}$ be its body. 
Since $w'$ is adjacent to $u_1$ and $w''$ is adjacent to $u_q$, there must exist a sub-path $U'=\{u_i,\ldots,u_{i+t}\}$
of $U$ such that $u_i$ is the only neighbour of $w'$ in $U'$ and $u_{i+t}$ is the only neighbour of $w''$ in $U'$.
We assume that $w',w'',U'$ are chosen so that $t$ (the length of the path $U'$) is as small as possible. This implies, in particular,
that no left wing has a neighbour in $U'$ other than $u_i$ and no right wing has a neighbour in $U'$ other than $u_{i+t}$. 

Assume now $t\ge s$. If $i=1$, we define $u_{i-1}$ to be the left wing different from $w'$, 
and if $i+t=q$, we define $u_{i+t+1}$ to be the right wing different from $w''$. If
$w'$ is adjacent to $w''$ or $w'$ is adjacent to $u_{i+t+1}$ or $w''$ is adjacent to $u_{i-1}$
or $u_{i-1}$ is adjacent to $u_{i+t+1}$,
then a chordless cycle of length at least $s+1$ arises. Otherwise, the vertices $w',w'',u_{i-1},u_i,\ldots,u_{i+t},u_{i+t+1}$
induce  a canonical graph of order at least $s$.
\end{proof}

\begin{lemma}\label{lem:to-dense-rake}
Let $k$ and $s$ be natural numbers. Every graph containing a $k+2$-rake as a subgraph contains 
either 
\begin{itemize}
\item an $s+5$-dense $k$-rake a subgraph or 
\item a canonical graph of order at least $s$ as an induced subgraph.
\end{itemize}
\end{lemma}

\begin{proof}
Consider a graph $G$ containing a $k$-rake $R$ as a subgraph.
For our construction it is essential that the second and second last vertices of the base of $R$
are roots while the first and the last vertices are not. To establish this condition we remove the teeth
whose roots are the first or the last vertices and possibly shorten the base so that it would start just
before the second root and end just after the second last root. This is where $k+2$ comes from.
After this preprocessing, we proceed as follows.

First, we transform any path between any two consecutive root vertices into a shortest, and hence a chordless, path
by cutting along any possible chords. Now any two consecutive root vertices together with their teeth, with the path
connecting them and with two other their neighbours in the base of $R$ form an $H$-graph satisfying conditions
of Lemma~\ref{lem:shorten}. If one of these $H$-graphs contains an induced canonical subgraph of order at least $s$,
the lemma is proved. Therefore, we assume that the wings of each of these graphs are connected by 
a short path as in (2) of Lemma~\ref{lem:shorten}. We now concatenate (glue) all these paths into the base of a new rake as follows.

Consider three consecutive vertices $u_{i-1},u_i,u_{i+1}$ in the base of $R$ with $u_i$ being a root vertex but \emph{not the first one}.
Let $v_i$ be the tooth of $u_i$. Also, denote by $P^l$ a short path connecting 
two wings of the $H$-graph on the left of $u_i$, and by $P^r$ a respective short path in the $H$-graph on the right of $u_i$. 
To simplify the discussion, we will assume that if $P^r$ starts at $u_{i-1}$, then its next vertex is neither $u_i$ nor $u_{i+1}$,
since otherwise we can transform $P^r$ by starting it at $v_i$, which will increase the length of the path by at most 1.
Also, we will assume that if $P^r$ starts at $v_i$, then its next vertex is not $u_{i+1}$,
since otherwise we can transform $P^r$ by adding $u_i$ between $v_i$ and $u_{i+1}$, which will increase the length of the path by at most 1.
We apply similar (symmetric) assumptions with respect to $P^l$. With these assumptions in mind, we now do the following. 
\begin{itemize} 
\item If both $P^l$ and $P^r$ contain $u_i$, then both of them start at $v_i$ (according to the above assumption).
In this case, we glue the two paths at $u_i$, define it to be a root vertex in the new rake and define $v_i$ to be its tooth.  
\item Assume that, say, $P^l$ contains $u_i$ (implying it contains $v_{i}$), while $P^r$ does not.
      Assume in addition that $P_l$ contains $u_{i-1}$. 
\begin{itemize}
\item If $P^r$ starts at $u_{i-1}$, 
then we glue the two paths at $u_{i-1}$ (by cutting $u_i$ and $v_i$ off $P^l$), define $u_{i-1}$ 
to be a root vertex and $u_i$ to be its tooth in the new rake. 
\item If $P^r$ starts at $v_{i}$, then we glue the two paths at $v_{i}$, define $u_{i}$ 
to be a root vertex and $u_{i+1}$ to be its tooth in the new rake. 
\end{itemize}
\item The same as in the previous case with the only difference that $P^l$ does not contain $u_{i-1}$,
\begin{itemize}
\item If $P^r$ starts at $u_{i-1}$, 
then we replace $v_i$ by $u_{i-1}$ in $P^l$,
glue the two paths at $u_{i-1}$, define $u_{i}$ 
to be a root vertex and $v_i$ to be its tooth in the new rake. 
\item If $P^r$ starts at $v_{i}$, then (like in the previous case) we glue the two paths at $v_{i}$, define $u_{i}$ 
to be a root vertex and $u_{i+1}$ to be its tooth in the new rake. 
\end{itemize}
\item Assume that neither $P^l$ nor $P^r$ contains $u_i$, then we distinguish between the following cases.
\begin{itemize}
\item If both paths start at $v_i$, then we glue them at $v_i$, define it to be a root vertex and $u_i$ its tooth in the new rake.
\item If one  of them, say $P^l$, starts at $v_i$, and the other one, that is $P^r$, starts at $u_{i-1}$, then we 
concatenate them by adding $u_i$ (which is adjacent to both $v_i$ and $u_{i-1}$),  define $u_i$ to be a root vertex and $u_{i+1}$
its tooth in the new rake.
\item If $P^l$ starts at $u_{i+1}$ $P^r$ starts at $u_{i-1}$, then we again  
concatenate them by adding $u_i$,  define $u_i$ to be a root vertex and $v_{i}$
its tooth in the new rake.
\end{itemize}
\end{itemize}

The procedure outlined above creates a new rake with $k$ teeth. The length of each
path used in the construction is initially at most $s+1$. In order to incorporate the assumptions
regarding $P^l$ and $P^r$ we increase them by at most $1$ on each end, so the resulting length is at most
$s+3$. Finally, the process of assignment of roots may require further increase by at most $1$ on each 
end. Hence, we conclude that the new rake is $s+5$-dense.
\end{proof}

\subsection{Dense rake subgraphs}
\label{sec:dense}

\begin{lemma}\label{lem:dense}
For every $s,q$ and $\ell$, there is a number $D=D(s,q,\ell)$ such that 
every graph containing an $\ell$-dense $D$-rake as a subgraph contains 
either 
\begin{itemize}
\item a canonical graph of order at least $s$ as an induced subgraph or
\item a biclique of order $q$ as a subgraph.
\end{itemize}
\end{lemma}

\begin{proof}
To define the number $D=D(s,q,\ell)$, we introduce intermediate notations
as follows: $b:=2(q-1)s^q+2sq+4$ and $c:=R(2,2,\max(b,2q))$,
where $R$ is the Ramsey number. With these notations the number $D$ is defined
as follows: $D=D(s,q,\ell):=Z(\ell c^2,2q,q)$, where $Z$ is the 
number defined in Theorem~\ref{thm:main}. 

Consider a graph $G$ containing an $\ell$-dense $D$-rake $R^0$ as a subgraph.
The base of this rake is a path $P^0$ of length at least $D$ and hence,
by Theorem~\ref{thm:main}, the base contains either a biclique 
of order at least $q$ as a subgraph (in which case we are done) or 
an {\it induced} path $P$ of length at least $\ell c^2$. Let 
us call any (inclusionwise) maximal sequence of consecutive vertices of $P^0$
that belong to $P$ a {\it block}. Assume the number of blocks is 
more than $c$. 
Let $P'$ be the subpath of $P$ induced by the first $c$ blocks.
Let $w_1, \dots, w_c$ be the rightmost vertices of the blocks.
Let $v_1, \dots, v_c$ be the vertices such that each $v_i$ is the vertex of $P_0$
immediately following $w_i$. Then $P'$ together with $v_1, \dots, v_c$ create
a $c$-rake with $P'$ being the induced base, $v_1, \dots, v_c$ being the teeth and 
$w_1, \dots, w_c$ being the respective roots.
If the number of blocks is at most $c$, 
then $P^0$ must contain a block of size at least $\ell c$, in which case this block 
also forms an induced base of a $c$-rake (since $R^0$ is $\ell$-dense). 
We see that in either case $G$ has a $c$-rake with an induced base.  
According to the definition of $c$, the $c$ teeth of this rake induce a graph which has either 
a clique of size $2q$ (and hence a biclique of order $q$ in which case we are done),
or an independent set of size $b$. By ignoring the teeth outside this  set we obtain 
a $b$-rake $R$ with an induced base and with teeth forming an independent set.

Let us denote the base of $R$ by $U$, its vertices by $u_1,\ldots,u_m$ (in the order of their
appearances in the path), and the teeth of $R$ by $t_1,\ldots,t_b$ (following 
the order of their root vertices). 

Denote $r:=(q-1)s^q+2$ and consider two sets of teeth $T_1=\{t_2,t_3,\ldots,t_r\}$ and $T_2=\{t_{b-1},t_{b-2},\ldots, t_{b-r+1}\}$.
By definition of $r$ and $b$, there are $2sq$ other teeth between $t_r$ and $t_{b-r+1}$,
and hence there is a set $M$ of $2sq$ consecutive vertices of $U$ between the root of $t_r$ and 
the root of $t_{b-r+1}$. We partition 
$M$ into $2q$ subsets (of consecutive vertices of $U$) of size $s$ each 
and for $i=1,\ldots,2q$ denote the $i$-th subset by $M_i$. 

If each vertex of $T_1$ has a neighbour in each of the first $q$ sets $M_i$, then by the Pigeonhole Principle 
there is a biclique of order $q$ with $q$ vertices in $T_1$ and $q$ vertices in $M$ (which can be proved by
analogy with Lemma~\ref{lem:grid}). Similarly, a biclique of order $q$ arises 
if  each vertex of $T_2$ has a neighbour in each of the last $q$ sets $M_i$. 
Therefore, we assume that there are two vertices $t_a\in T_1$ and $t_b\in T_2$ and 
two sets $M_x$ and $M_y$ with $x<y$ such that $t_a$ has no neighbours in $M_x$, while 
$t_b$ has no neighbours in $M_y$. 

By definition, $t_a$ has a neighbour in $U$ (its root) on the left of $M_x$. If additionally
$t_a$ has a neighbour to the right of $M_x$, then a chordless cycle of length at least 
$s$ arises (since $|M_x|=s$ and $t_a$  has no neighbours in $M_x$), in which case the lemma is true.
This restricts us to the case, when all neighbours of $t_a$ in $U$ are located to the left of $M_x$.
By analogy, we assume that all neighbours of $t_b$ in $U$ are located to the right of $M_y$.
Let $u_i$ be the rightmost neighbour of $t_a$ in $U$ and $u_j$ be the leftmost neighbour of $t_b$ in $U$.
According to the above discussion, $i<j$ and $j-j>2s$. But then the vertices $t_a,t_b, u_{i-1}, u_i,\ldots,u_j,u_{j+1}$
induce an $H$-graph (possibly tight or semi-tight) of order more than $s$ (the existence of vertices $u_{i-1}$ and $u_{j+1}$
follows from the fact that $T_1$ does not include $t_1$, while $T_2$ does not include $t_b$).          
\end{proof}

\subsection{Canonical graphs and bicliques in graphs of large tree-width}
\label{sec:ltw}

\begin{theorem}\label{thm:prefinal}
For every $s,q$, there is a number $X=X(s,q)$ such that 
every graph of tree-width at least $X$ contains 
either 
\begin{itemize}
\item a canonical graph of order at least $s$ as an induced subgraph or
\item a biclique of order $q$ as a subgraph.
\end{itemize}
\end{theorem}

\begin{proof}
We define $X(s,q)$ as $X(s,q):=f(D(s,q,s+5)+2)$, where $f$ comes from Lemma~\ref{lem:tw-rake}
and $D$ comes from Lemma~\ref{lem:dense}.
If a graph $G$ has tree-width at least $X(s,q)$, then by Lemma~\ref{lem:tw-rake}
it contains $D(s,q,s+5)+2$-rake $R$ as a subgraph. 
Then, by Lemma~\ref{lem:to-dense-rake}, 
$G$ contains either a canonical graph of order at least $s$ as an induced subgraph,
or an $s+5$-dense $D(s,q,s+5)$-rake as a subgraph. In the first case, the theorem is proved.
In the second case, we conclude by Lemma~\ref{lem:dense} that 
$G$ contains either a canonical graph of order at least $s$ as an induced subgraph or
a biclique of order $q$ as a subgraph. 
\end{proof}

\subsection{Proof of Theorem~\ref{thm:main-wqo}}
\label{sec:summary}
Let $\cal Y$ be a hereditary class of graphs with $K_{q,q}$ and $K_r$ contained  
in the set of forbidden induced subgraphs
and assume $\cal Y$ is well-quasi-ordered by the induced subgraph relation. 

Suppose by contradiction that $\cal Y$
contains an infinite sequence ${\cal Y}'$ of graphs of increasing tree-width. In this sequence, there 
must exists a graph $G^1$ of tree-width at least $X(s,q)$, where $X(s,q)$ is defined in Theorem~\ref{thm:prefinal} and 
$s$ is an arbitrarily chosen constant.
Then by Theorem~\ref{thm:prefinal} $G^1$ contains a canonical graph $H^1$ of order at least $s$. We denote 
the order of $H^1$ by $s_1$ and find in ${\cal Y}'$ a graph $G^2$ of tree-width at least $X(s_1+1,q)$.
$G^2$ must contains a canonical graph $H^2$ of order $s_2\ge s_1+1$, and so on. 
In this way, we construct an infinite sequence $H^1,H^2,\ldots$, which form an antichain by Claim~\ref{claim:antichain}.
This contradicts the assumption that $\cal Y$ is well-quasi-ordered by the induced subgraph relation and 
hence shows that $\cal Y$ is of bounded tree-width.

\bibliographystyle{plain}
\bibliography{width}

\begin{thebibliography}{10}

\bibitem{speed}
J.~Balogh, B.~Bollobas, and D.~Weinreich.
\newblock The speed of hereditary properties of graphs.
\newblock {\em Journal of Combinatorial Theory, Series B}, 79(2):131--156,
  2000.

\bibitem{CorMakRotics}
Bruno Courcelle, Johann~A. Makowsky, and Udi Rotics.
\newblock Linear time solvable optimization problems on graphs of bounded
  clique-width.
\newblock {\em Theory Comput. Syst.}, 33(2):125--150, 2000.

\bibitem{DRT10}
Jean Daligault, Micha{\"e}l Rao, and St{\'e}phan Thomass{\'e}.
\newblock Well-quasi-order of relabel functions.
\newblock {\em Order}, 27(3):301--315, 2010.

\bibitem{Ding09}
Guoli Ding.
\newblock On canonical antichains.
\newblock {\em Discrete Mathematics}, 309(5):1123--1134, 2009.

\bibitem{Fellows89}
M.~R. Fellows and M.~A. Langston.
\newblock On search decision and the efficiency of polynomial-time algorithms.
\newblock In {\em Proceedings of the twenty-first annual ACM symposium on
  Theory of computing}, STOC '89, pages 501--512, New York, NY, USA, 1989. ACM.

\bibitem{Finkel}
Alain Finkel and Ph. Schnoebelen.
\newblock Well-structured transition systems everywhere!
\newblock {\em Theor. Comput. Sci.}, 256(1-2):63--92, 2001.

\bibitem{GurWan}
Frank Gurski and Egon Wanke.
\newblock The tree-width of clique-width bounded graphs without {\it
  k$_{\mbox{n, n}}$}.
\newblock In {\em WG}, pages 196--205, 2000.

\bibitem{Kruskal}
Joseph~B. Kruskal.
\newblock The theory of well-quasi-ordering: A frequently discovered concept.
\newblock {\em J. Comb. Theory, Ser. A}, 13(3):297--305, 1972.

\bibitem{minor-closed-small}
Serguei Norine, Paul Seymour, Robin Thomas, and Paul Wollan.
\newblock Proper minor-closed families are small.
\newblock {\em J. Comb. Theory Ser. B}, 96(5):754--757, 2006.

\bibitem{Zaran}
T.~K\H ov\'ari, V.T. S\'os, and P.~Tur\'an.
\newblock On a problem of {K}. {Z}arankiewicz.
\newblock {\em Colloq. Math.}, 3:50--57, 1954.

\bibitem{minor-wqo}
Neil Robertson and P.~D. Seymour.
\newblock Graph minors {X}{X}. {W}agner's conjecture.
\newblock {\em J. Comb. Theory Ser. B}, 92(2):325--357, 2004.

\bibitem{RST94}
Neil Robertson, Paul Seymour, and Robin Thomas.
\newblock Quickly excluding a planar graph.
\newblock {\em J. Comb. Theory Ser. B}, 62(2):323--348, 1994.

\end{thebibliography}

\end{document}